\documentclass[a4paper,11pt,reqno]{amsart}
\usepackage{amsmath}
\usepackage{amstext}
\usepackage{amsbsy}
\usepackage{amsopn}
\usepackage{upref}
\usepackage{amsthm}
\usepackage{amsfonts}
\usepackage{amssymb}
\usepackage{mathrsfs}
\usepackage{times}
\usepackage{cite}
\allowdisplaybreaks
 \newtheorem{theorem}{Theorem}

 \newtheorem{lemma}[theorem]{Lemma}
\theoremstyle{definition}
 
\theoremstyle{remark}
 
\begin{document}
\title{Holomorphic Hermite functions in Segal-Bargmann spaces}
\author[H.~Chihara]{Hiroyuki Chihara}
\address{College of Education, University of the Ryukyus, Nishihara, Okinawa 903-0213, Japan}
\email{aji@rencho.me}
\thanks{Supported by the JSPS Grant-in-Aid for Scientific Research \#16K05221.}
\begin{abstract}
We study systems of holomorphic Hermite functions in the Segal-Bargmann spaces, 
which are Hilbert spaces of entire functions on the complex Euclidean space, 
and are determined by the Bargmann-type integral transform on the real Euclidean space. 
We prove that for any positive parameter 
which is strictly smaller than the minimum eigenvalue of 
the positive Hermitian matrix associated with the transform, 
one can find a generator of holomorphic Hermite functions 
whose anihilation and creation operators satisfy canonical commutation relations. 
In other words, we find the necessary and sufficient conditions 
so that some kinds of entire functions can be such generators. 
Moreover, we also study the complete orthogonality, 
the eigenvalue problems and the Rodrigues formulas. 
\end{abstract}
\maketitle
\section{Introduction}
\label{section:introduction}
We study families of holomorphic Hermite functions of $n$-variables in the framework of (generalized) Segal-Bargmann spaces which are reproducing kernel Hilbert spaces of entire functions and are determined by Bargmann-type integral transforms on the $n$-dimensional real Euclidean space. More precisely, we fix an arbitrary Segal-Bragmann space and investigate conditions so that a family of holomorphic Hermite functions has desirable properties: orthonormality, canonical commutation relations of anihilation and creation operators, completeness, and etc. 
\par
Firstly we introduce Segal-Bargmann spaces associated with Bargmann-type integral transforms and review their basic properties quickly. 
These are originated by Sj\"ostrand. 
He introduced these to develop microlocal analysis. 
See \cite{sjoestrand} or \cite{chihara1}. 
Let $n$ be a positive integer describing the space dimension. 
Let $A$, $B$ and $C$ be $n \times n$ complex matrices satisfying 
\begin{equation}
{}^tA=A, 
\qquad
\det{B}\ne0, 
\qquad
{}^tC=C, 
\quad
C_I
:=
\frac{C-\bar{C}}{2\sqrt{-1}}>0. 
\label{equation:matrices} 
\end{equation}
Set 
$$
\langle{z,\zeta}\rangle
:=
z_1\zeta_1+\dotsb+z_n\zeta_n, 
\quad
\lvert{z}\rvert^2
:=
z_1\overline{z_1}+\dotsb+z_n\overline{z_n}
\quad
\text{for}
\quad
z
=
\begin{bmatrix}
z_1 \\ \vdots \\ z_n 
\end{bmatrix}, 
\zeta
=
\begin{bmatrix}
\zeta_1 \\ \vdots \\ \zeta_n 
\end{bmatrix}
\in
\mathbb{C}^n.
$$
Roughly speaking, the generalized Bargmann transform 
is a global Fourier integral operator on $\mathbb{R}^n$ 
with a complex phase function of the form 
$$
\phi(z,x)
=
\frac{1}{2}
\langle{z,Az}\rangle
+
\langle{z,Bx}\rangle
+
\frac{1}{2}
\langle{x,Cx}\rangle, 
\qquad
z\in\mathbb{C}^n, 
\quad
x\in\mathbb{R}^n.
$$
The Schwartz class on $\mathbb{R}^n$ is denoted by $\mathscr{S}(\mathbb{R}^n)$. 
More precisely, 
the Bargmann-type transform of $u\in\mathscr{S}(\mathbb{R}^n)$ is defined by 
\begin{equation}
Tu(z)
:=
C_\phi
\int_{\mathbb{R}^n}
e^{\sqrt{-1}\phi(z,x)}
u(x)
dx, 
\quad
z\in\mathbb{C}^n,  
\label{equation:bargmann}
\end{equation}
$$
C_\phi
=
2^{-n/2}
\pi^{-3n/4}
\lvert\det{B}\rvert
(\det{C_I})^{-1/4}. 
$$
For any fixed $z\in\mathbb{C}^n$, 
there exists a constant $\varepsilon>0$ such that  
$$
\operatorname{Re}\bigl\{\sqrt{-1}\phi(z,x)\bigr\}
=
\mathcal{O}(1+\lvert{x}\rvert)
-
\frac{1}{2}
\langle{x,C_Ix}\rangle
\leqslant
-\varepsilon\lvert{x}\rvert^2 
$$
for sufficiently large $x\in\mathbb{R}^n$. 
Then the Bargmann transform can be extended for tempered distributions on $\mathbb{R}^n$. 
Using this again, we can get the following function 
$$
\Phi(z)
:=
\max_{x\in\mathbb{R}^n}
\operatorname{Re}\bigl\{\sqrt{-1}\phi(z,x)\bigr\}
=
\langle{z,\Phi^{\prime\prime}_{z\bar{z}}\bar{z}}\rangle
+
\operatorname{Re}\langle{z,\Phi^{\prime\prime}_{zz}z}\rangle, 
$$
$$
\Phi^{\prime\prime}_{z\bar{z}}
=
\frac{BC_I^{-1}B^\ast}{4},
\quad
\Phi^{\prime\prime}_{zz}
=
-
\frac{BC_I^{-1}{}^tB}{4}
-
\frac{A}{2\sqrt{-1}}. 
$$
Note that $\Phi^{\prime\prime}_{z\bar{z}}$ is a positive definite Hermitian matrix, 
and that $\Phi^{\prime\prime}_{zz}$ is a complex symmetric matrix. 
\par
Here we introduce function spaces. 
We denote by $L^2(\mathbb{R}^n)$ 
the set of all square-integrable functions on $\mathbb{R}^n$. 
This is a Hilbert space equipped with the standard inner product and the norm 
\begin{align*}
  (f,g)_{L^2(\mathbb{R}^n)}
& :=
  \int_{\mathbb{R}^n}
  f(x)
  \overline{g(x)}
  dx, 
\\
  \lVert{f}\rVert_{L^2(\mathbb{R}^n)}
& :=
  \sqrt{(f,f)_{L^2(\mathbb{R}^n)}}, 
  \quad
  f,g \in L^2(\mathbb{R}^n).  
\end{align*}
Let $L(dz)$ be the Lebesgue measure for $z\in\mathbb{C}^n\simeq\mathbb{R}^{2n}$. 
We denote by $L^2_\Phi(\mathbb{C}^n)$ the set of all square-integrable functions on $\mathbb{C}^n$ with respect to a weighted measure $e^{-2\Phi(z)}L(dz)$. 
This is also a Hilbert space equipped with an inner product and the norm 
\begin{align*}
  (F,G)_{L^2_\Phi(\mathbb{C}^n)}
& :=
  \int_{\mathbb{C}^n}
  F(z)
  \overline{G(z)}
  e^{-2\Phi(z)}L(dz), 
\\
  \lVert{F}\rVert_{L^2_\Phi(\mathbb{C}^n)}
& :=
  \sqrt{(F,F)_{L^2_\Phi(\mathbb{C}^n)}}, 
  \quad
  F,G \in L^2_\Phi(\mathbb{C}^n). 
\end{align*}
Let $\operatorname{Hol}(\mathbb{C}^n)$ 
be the set of all entire functions on $\mathbb{C}^n$. 
A (generalized) Segal-Bargmann space $\mathscr{H}_\Phi(\mathbb{C}^n)$ 
is defined by 
$\mathscr{H}_\Phi(\mathbb{C}^n):=\operatorname{Hol}(\mathbb{C}^n) \cap L^2_\Phi(\mathbb{C}^n)$. 
It is easy to see that $\mathscr{H}_\Phi(\mathbb{C}^n)$ is a closed subspace of $L^2_\Phi(\mathbb{C}^n)$ and becomes also a Hilbert space. Set 
$$
(F,G)_{\mathscr{H}_\Phi(\mathbb{C}^n)}
:=
(F,G)_{L^2_\Phi(\mathbb{C}^n)}, 
\quad
\lVert{F}\rVert_{\mathscr{H}_\Phi(\mathbb{C}^n)}
:=
\lVert{F}\rVert_{L^2_\Phi(\mathbb{C}^n)},
\quad
F,G \in \mathscr{H}_\Phi(\mathbb{C}^n). 
$$
\par
It is known that the Bargmann-type transform $T$ gives a Hilbert space isomorphism of $L^2(\mathbb{R}^n)$ onto $\mathscr{H}_\Phi(\mathbb{C}^n)$, that is, 
$$
(Tf,Tg)_{\mathscr{H}_\Phi(\mathbb{C}^n)}
=
(f,g)_{L^2(\mathbb{R}^n)}, 
\quad
f,g \in L^2(\mathbb{R}^2). 
$$
The inverse mapping of $T$ is given by 
$$
T^\ast F(x)
:=
C_\phi
\int_{\mathbb{C}^n}
e^{-\sqrt{-1}\ \overline{\phi(z,x)}}
F(z)e^{-2\Phi(z)}L(dz),
\quad
F \in \mathscr{H}_\Phi(\mathbb{C}^n),
\quad
x\in\mathbb{R}^n. 
$$
Note that $T^\ast$ is also defined for $F \in L^2_\Phi(\mathbb{C}^n)$. 
Moreover $T \circ T^\ast$ becomes an orthogonal projection of 
$L^2_\Phi(\mathbb{C}^n)$ onto $\mathscr{H}_\Phi(\mathbb{C}^n)$, 
which is concretely described as 
$$
T \circ T^\ast F(z)
=
C_\Phi
\int_{\mathbb{C}^n}
e^{2\Psi(z,\bar{\zeta})}
F(\zeta)
e^{-2\Phi(\zeta)}
L(d\zeta),
\quad
F \in L^2_\Phi(\mathbb{C}^n),
$$
$$
C_\Phi
=
(2\pi)^{-n}
\lvert{\det{B}}\rvert^2
(\det C_I)^{-1}, 
$$ 
$$
\Psi(z,\bar{\zeta})
=
\langle{z,\Phi^{\prime\prime}_{z\bar{z}}\bar{\zeta}}\rangle
+
\frac{1}{2}
\langle{z,\Phi^{\prime\prime}_{zz}z}\rangle
+
\frac{1}{2}
\langle{\bar{\zeta},\overline{\Phi^{\prime\prime}_{zz}}\bar{\zeta}}\rangle.
$$
Note that $\Phi(z)=\Psi(z,\bar{z})$. 
The significant property is 
$F=T \circ T^\ast F$ for $F \in \mathscr{H}_\Phi(\mathbb{C}^n)$. 
Thus, $\mathscr{H}_\Phi(\mathbb{C}^n)$ is a reproducing kernel Hilbert space 
with a kernel $K_\Phi(z,\zeta)=C_\Phi e^{2\Psi(z,\bar{\zeta})}$. 
See \cite{sjoestrand} for more detail about the above. 
\par
Here we show the most important example of 
Bargmann-type transforms and Segal-Bargmann spaces, those are,  
the standard Bargmann transform and the standard Segal-Bargmann space 
introduced by Bargmann in \cite{bargmann1} and \cite{bargmann2}. 
The standard Bargmann transform is 
$$
T_Bu(z)
=
2^{-n/2}
\pi^{-3n/4}
\int_{\mathbb{C}^n}
e^{-(\langle{z,z}\rangle/4-\langle{z,x}\rangle+\langle{x,x}\rangle/2)}
u(x)
dx,
$$
that is, 
the phase function is 
$$
\phi_B(z,x)
=
\frac{\sqrt{-1}}{4}
\langle{z,z}\rangle
-
\sqrt{-1}
\langle{z,x}\rangle
+
\frac{\sqrt{-1}}{2}
\langle{x,x}\rangle, 
$$
$$
A=\frac{\sqrt{-1}}{2}E, 
\quad
B=-\sqrt{-1}E, 
\quad
C=\sqrt{-1}E,
$$
where $E$ is the $n{\times}n$ identity matrix. 
In this case we have 
$\Phi_B(z)=\lvert{z}\rvert^2/4$ 
and 
$\Psi_B(z,\bar{\zeta})=-\langle{z,\bar{\zeta}}\rangle/4$. 
The standard Segal-Bargmann space $\mathscr{H}_B(\mathbb{C}^n)$ is defined by  
$$
\mathscr{H}_B(\mathbb{C}^n)
=
\operatorname{Hol}(\mathbb{C}^n) \cap L^2_B(\mathbb{C}^n), 
\quad
L^2_B(\mathbb{C}^n)
=
L^2(\mathbb{C}^n, e^{-\lvert{z}\rvert^2/2}L(dz)), 
$$
and its reproducing formula is given by 
$$
F(z)
=
\frac{1}{(2\pi)^n}
\int_{\mathbb{C}^n}
e^{\langle{z,\bar{\zeta}}\rangle/2}
F(\zeta)
e^{-\lvert\zeta\rvert^2/2}
L(d\zeta), 
\quad
F \in \mathscr{H}_B(\mathbb{C}^n). 
$$
See \cite{folland}
\par
Secondly we review holomorphic Hermite functions appearing in mathematics studies. 
In 1990, van Eijndhoven and Meyers first introduced 
one-dimensional holomorphic Hermite functions in \cite{EM}. 
Let $s$ be a parameter satisfying $0<s<1$. 
They introduced a Segal-Bargmann space $\chi_s(\mathbb{C})$ by 
$$
\chi_s(\mathbb{C})
=
\operatorname{Hol}(\mathbb{C})
\cap
L^2
\left(
\mathbb{C}, 
\exp\left(
    -\frac{1-s^2}{2s}\lvert{z}\rvert^2
    +
    \frac{1+s^2}{4s}
    (z^2+\bar{z}^2)
    \right)
L(dz)
\right), 
$$
and a family of holomorphic Hermite functions $\{\psi_k^s(z)\}_{k=0}^\infty$ 
given by a Rodrigues formula 
$$
\psi^s_k(z)
=
\frac{(1-s)^{1/2}}{2^{k/2}\pi^{1/2}s^{1/4}k!^{1/2}}
\left(\frac{1-s}{1+s}\right)^{k/2}
(-1)^k
e^{z^2/2}
\left(\frac{d}{dz}\right)^k
e^{-z^2}, 
\quad
k=0,1,2,\dotsc, 
$$
They studied basic properties of $\{\psi_k^s(z)\}_{k=0}^\infty$ 
including its orthonormality in $\chi_s(\mathbb{C})$. 
After that, the function space $\mathscr{X}_s(\mathbb{C})$ 
and the system of holomorphic Hermite functions $\{\psi^s_k\}_{k=0}^\infty$ 
have been applied to studying 
quantization on $\mathbb{C}$ and related problems 
(see \cite{szafraniec, GF, AGHS}), 
combinatorics and counting (see \cite{IS}) and etc. 
We call $\psi_0^s(z)=\pi\sqrt{s}e^{-z^2/2}/(1-s)$ 
the generator of the family $\{\psi_n^s\}_{n=0}^\infty$. 
Recently, in \cite{chihara3} the author explains this material 
in terms of the Bargmann-type integral transforms and the Segal-Bargmann spaces. 
See also \cite{chihara2}. 
\par
Quite recently, in \cite{GHS} G\'orska, Horzela and Szafraniec introduced 
a two-dimensional version of \cite{EM}. They consider a Segal-Bargmann space 
$$
\chi_s(\mathbb{C}^2)
=
\operatorname{Hol}(\mathbb{C}^2)
\cap 
L^2
\left(
\mathbb{C}^2,
\exp\left(
    -
    \frac{1-s^2}{4s}(\lvert{z}\rvert^2+\lvert{\zeta}\rvert^2)
    +
    \frac{1+s^2}{4s}(z\zeta+\bar{z}\bar{\zeta})
    \right)
    L(dz)L(d\zeta)
\right), 
$$
and introduced an orthonormal basis $\{\psi_{k,l}^s\}_{k,l=0}^\infty$ 
given by a Rodrigues formula 
$$
\psi_{k,l}^s(z,\zeta)
=
\frac{1-s}{\pi\sqrt{sk!l!}}
\left(\frac{1-s}{1+s}\right)^{(k+l)/2}
(-1)^{k+l}
e^{z\zeta/2}
\left(\frac{\partial}{\partial \zeta}\right)^k
\left(\frac{\partial}{\partial z}\right)^l
e^{-z\zeta}. 
$$
We also call $\psi_{0,0}^s$ the generator of $\{\psi_{k,l}^s\}_{k,l=0}^\infty$. 
\par
A question arises: What is the relationship between a Segal-Bargmann spaces and a family of holomorphic Hermite functions? 
Here we state our question precisely. 
Fix an arbitrary Segal-Bargmann space $\mathscr{H}_\Phi(\mathbb{C}^n)$, 
and consider a holomorphic function $\tilde{\psi}_0(z)\in\mathscr{H}_\Phi(\mathbb{C}^n)$ of the form 
$$
\tilde{\psi}_0(z)
=
\exp\bigl(-\langle{z,Qz}\rangle\bigr), 
\quad
z\in\mathbb{C}^n,
$$
where $Q$ is an $n \times n$ complex symmetric matrix. 
Our question is as follows: 
\begin{quote}
What is the condition on $Q$ so that $\tilde{\psi}_0(z)$ can be a generator of holomorphic Hermite functions in $\mathscr{H}_\Phi(\mathbb{C}^n)$?  
\end{quote}
More precisely, a generator is required to have anihilation and creation operators which satisfy canonical commutation relations so that orthogonality and Rodrigues formulas hold.
\par
The purpose of the present paper is to answer our question above. 
Roughly speaking, our main results in the present paper are the necessary and sufficient conditions on $Q$, and the freedom of $Q$ is described by the cardinality of the set of $n \times n$ unitary matrices satisfying some conditions. 
\par
The organization of the present paper is as follows. In Section~2 we investigate the necessary and sufficient conditions on $Q$ so that the desirable canonical commutation relations hold, and introduce a family of holomorphic Hermite functions. In Section~3 we study some properties of the family: orthogonality, an eigenvalue problem for some Hamiltonian, Rodrigues formulas and the completeness in $\mathscr{H}_\Phi(\mathbb{C}^n)$. Finally in Section~4 we see examples given in \cite{EM} and \cite{GHS} from a point of view of our results in the present paper. 
%
%
\section{Definition}
\label{section:definition} 
Throughout of Sections~2 and 3, fix an arbitrary Bargmann transform \eqref{equation:bargmann}. In the present section we seek conditions on $Q$ so that $\tilde{\psi}_0(z)=\exp\bigl(-\langle{z,Qz}\rangle\bigr)$ belongs to $\mathscr{H}_\Phi(\mathbb{C}^n)$ and has anihilation and creation operators satisfying canonical commutation relations. 
We employ heuristic arguments by imposing decidable properties on $\tilde{\psi}_0$. 
\par
Firstly we seek the condition on $Q$ so that $\tilde{\psi}_0 \in \mathscr{H}_\Phi(\mathbb{C}^n)$. It suffices to require $\tilde{\psi}_0 \in L^2_\Phi(\mathbb{C}^n)$ since $\tilde{\psi}_0$ is an entire function on $\mathbb{C}^n$. 
Note that 
$$
\lvert\tilde{\psi}_0(z)\rvert^2e^{-2\Phi(z)}
=
\exp\bigl(
     -2
     \langle{z,\Phi^{\prime\prime}_{z\bar{z}}\bar{z}}\rangle
     -2
     \operatorname{Re}
     \langle{z,(\Phi^{\prime\prime}_{zz}+Q)z}\rangle
     \bigr). 
$$ 
It is natural to impose the following condition on $Q$.
\begin{quote}
{\bf Condition~1}. 
\quad
There exists a positive constant $\delta$ such that 
$$
\langle{z,\Phi^{\prime\prime}_{z\bar{z}}\bar{z}}\rangle
+
\operatorname{Re}
\langle{z,(\Phi^{\prime\prime}_{zz}+Q)z}\rangle
\geqslant
\delta
\lvert{z}\rvert^2,
\quad
z\in\mathbb{C}^n.
$$
\end{quote}
\par
To carry out our computations below, we introduce notation. Recall that $\Phi^{\prime\prime}_{z\bar{z}}$ is a positive definite Hermitian matrix. Then $\Phi^{\prime\prime}_{z\bar{z}}$ is invertible and $(\Phi^{\prime\prime}_{z\bar{z}})^{-1}$ is also a positive definite Hermitian matrix. Set 
$$
\Phi^{\prime\prime}_{z\bar{z}}=[\alpha_{ij}], 
\quad
\overline{\alpha_{ij}}=\alpha_{ji},  
\quad
(\Phi^{\prime\prime}_{z\bar{z}})^{-1}=[\beta_{ij}], 
\quad
\overline{\beta_{ij}}=\beta_{ji}, 
$$ 
$$
\det\bigl(tE-\Phi^{\prime\prime}_{z\bar{z}}\bigr)
=
\prod_{i=1}^n
(t-\lambda_i^2), 
\quad
\lambda_i>0, 
\quad
\lambda_0:=\min\{\lambda_1, \dotsc, \lambda_n\}>0. 
$$
Recall that $\Phi^{\prime\prime}_{zz}$ and $Q$ are complex symmetric matrices. 
Set 
$$
\Phi^{\prime\prime}_{zz}=[\gamma_{ij}], 
\quad
\gamma_{ij}=\gamma_{ji}, 
\quad
Q=[q_{ij}], 
\quad
q_{ij}=q_{ji}.
$$
\par
Secondly we consider anihilation and creation operators for $\tilde{\psi}_0$ in $\mathscr{H}_\Phi(\mathbb{C}^n)$. We shall construct anihilation $\Lambda_1, \dotsc, \Lambda_n$ of the form 
$$
\Lambda
=
\begin{bmatrix}
\Lambda_1
\\
\vdots
\\
\Lambda_n 
\end{bmatrix}
=
\frac{\partial}{\partial z}
+
Rz, 
$$
where $R=[r_{ij}]$ is a complex $n \times n$ matrix. 
In other words, 
$$
\Lambda_i
=
\frac{\partial}{\partial z_i}
+
\sum_{j=1}^nr_{ij}z_j, 
\quad
i=1,\dotsc,n. 
$$
We employ the adjoints $\Lambda_1^\ast, \dotsc, \Lambda_n^\ast$ in $\mathscr{H}_\Phi(\mathbb{C}^n)$ as creation operators for $\tilde{\psi}_0$. Moreover, we require functions generated by the creation operators to be eigenfunctions for a Hamiltonian 
$$
\Lambda_1^\ast\Lambda_1+\dotsb+\Lambda_n^\ast\Lambda_n
$$  
in the same way as the standard Hermite functions on the real Euclidean space 
and the Hamiltonian of the harmonic oscillators. 
For these reasons we impose the following conditions on $\Lambda$. 
\begin{quote}
{\bf Condition~2.}
\quad
\begin{itemize}
\item[(i)] 
$\Lambda_i\tilde{\psi}_0=0$ for $i=1,\dotsc,n$. 
\item[(ii)] 
$\bigl[\Lambda_i,\Lambda_j\bigr]=0$ for $i,j=1,\dotsc,n$, 
which implies 
$\bigl[\Lambda_i^\ast,\Lambda_j^\ast\bigr]=0$ for $i,j=1,\dotsc,n$. 
\item[(iii)] 
There exists a positive constant $\rho$ such that 
$\bigl[\Lambda_i,\Lambda_j^\ast\bigr]=2\rho^2\delta_{ij}$ for $i,j=1,\dotsc,n$, 
where $\delta_{ij}$ is the Kronecker's delta.  
\end{itemize}
\end{quote}
Here we call (ii) and (iii) of Condition~2 canonical commutation relations. 
The plan of our heuristic arguments in the present section are the following. 
Firstly we determine $R$ satisfying (i), 
and check that (ii) is automatically satisfied. 
Secondly we compute $\Lambda^\ast$. 
Thirdly we seek conditions so that (iii) is satisfied. 
Lastly we see that if $Q$ satisfies conditions for (i), (ii) and (iii), 
then Condition~1 is automatically satisfied. 
\par
We consider (i) and (ii) of Condition~2. 
\begin{lemma}
\label{theorem:lemma1}
{\rm (i)} of Condition~2 holds if and only if $R=2Q$, that is, 
$$
\Lambda=\frac{\partial}{\partial z}+2Qz. 
$$
In this case {\rm (ii)} of Condition~2 is automatically satisfied. 
\end{lemma}
\begin{proof}
We first check that $R=2Q$ is the necessary and sufficient condition for (i). 
For $i=1,\dotsc,n$, a simple computation gives 
\begin{align*}
  \Lambda_i\tilde{\psi}_0(z)
& =
  \left\{
  \frac{\partial}{\partial z_i}
  \bigl(
  -
  \langle
  z,Qz
  \rangle
  \bigr)
  +
  \sum_{j=1}^n
  r_{ij}z_j
  \right\}
  \tilde{\psi}_0(z)
\\
& =
  \left\{
  -
  \frac{\partial}{\partial z_i}
  \sum_{j=1}^n\sum_{k=1}^n
  q_{jk}z_jz_k
  +
  \sum_{j=1}^n
  r_{ij}z_j
  \right\}
  \tilde{\psi}_0(z)
\\
& =
  \left\{
  -
  \sum_{j=1}^n\sum_{k=1}^n
  q_{jk}\delta_{ij}z_k
  -
  \sum_{j=1}^n\sum_{k=1^n}
  q_{jk}z_j\delta_{ik}
  +
  \sum_{j=1}^n
  r_{ij}z_j
  \right\}
  \tilde{\psi}_0(z)
\\
& =
  \left\{
  -
  \sum_{k=1}^n
  q_{ik}z_k
  -
  \sum_{j=1}^n
  q_{ji}z_j
  +
  \sum_{j=1}^n
  r_{ij}z_j
  \right\}
  \tilde{\psi}_0(z)
\\
& =
  \left\{
  -
  \sum_{j=1}^n
  q_{ij}z_j
  -
  \sum_{j=1}^n
  q_{ij}z_j
  +
  \sum_{j=1}^n
  r_{ij}z_j
  \right\}
  \tilde{\psi}_0(z)
\\
& =
  \left\{
  \sum_{j=1}^n
  \bigl(r_{ij}-2q_{ij}\bigr)z_j
  \right\}
  \tilde{\psi}_0(z),
\end{align*}
where we used the symmetry of $Q$, that is, $q_{ji}=q_{ij}$. 
Here we suppose that (i) of Condition~2 holds. 
Since $\tilde{\psi}_0(z)\ne0$, we have 
$$
\sum_{j=1}^n
\bigl(r_{ij}-2q_{ij}\bigr)z_j
=
0, 
\quad
i=1,\dotsc,n 
$$
for all $z={}^t(z_1,\dotsc.z_n)\in\mathbb{C}^n$. 
Then we have $r_{ij}=2q_{ij}$ for $i,j=1,\dotsc,n$, that is, $R=2Q$. 
It is easy to see that if $R=2Q$, then (i) holds. 
\par
Next we see that if $R=2Q$, then (ii) of Condition~2 is satisfied. 
Let $F$ be a holomorphic function on $\mathbb{C}^n$. 
An elementary computation gives
\begin{align*}
  \bigl[\Lambda_i,\Lambda_j\bigr]F(z)
& =
  \left(
  \frac{\partial}{\partial z_i}
  +
  2\sum_{k=1}^nq_{ik}z_k
  \right) 
  \left(
  \frac{\partial}{\partial z_j}
  +
  2\sum_{l=1}^nq_{jl}z_l
  \right) 
  F(z)
\\
& -
  \left(
  \frac{\partial}{\partial z_j}
  +
  2\sum_{l=1}^nq_{jl}z_l
  \right) 
  \left(
  \frac{\partial}{\partial z_i}
  +
  2\sum_{k=1}^nq_{ik}z_k
  \right) 
  F(z)
\\
& =
  2
  \sum_{l=1}^nq_{jl}
  \left(\frac{\partial}{\partial z_i}z_l\right)
  F(z)
  -
  2
  \sum_{k=1}^nq_{ik}
  \left(\frac{\partial}{\partial z_j}z_k\right)
  F(z)
\\
& =
  2
  \sum_{l=1}^nq_{jl}\delta_{il}
  F(z)
  -
  2
  \sum_{k=1}^nq_{ik}\delta_{jk}
  F(z)
\\
& =
  2(q_{ji}-q_{ij})
  F(z)
  =
  2(q_{ij}-q_{ij})
  F(z)
  =
  0, 
\end{align*}
which is desired. 
This completes the proof. 
\end{proof}
We compute the adjoint $\Lambda^\ast$. 
\begin{lemma}
\label{theorem:lemma2} 
We have 
$$
\Lambda^\ast
=
\bigl(\Phi^{\prime\prime}_{zz}+Q\bigr)^\ast
\bigl(\Phi^{\prime\prime}_{z\bar{z}}\bigr)^{-1}
\frac{\partial}{\partial z}
+
2
\Bigl\{
\overline{\Phi^{\prime\prime}_{z\bar{z}}}
-
\bigl(\Phi^{\prime\prime}_{zz}+Q\bigr)
\bigl(\Phi^{\prime\prime}_{z\bar{z}}\bigr)^{-1}
\Phi^{\prime\prime}_{zz}
\Bigr\}
z, 
$$
that is, for $i=1,\dotsc,n$,  
$$
\Lambda_i^\ast
=
\sum_{j=1}^n
\sum_{k=1}^n
\bigl(\overline{\gamma_{ij}}+\overline{q_{ij}}\bigr)
\beta_{jk}
\frac{\partial}{\partial z_k}
+
2\sum_{l=1}^n
\Bigl\{
\overline{\alpha_{il}}
-
\sum_{j=1}^n
\sum_{k=1}^n
\bigl(\overline{\gamma_{ij}}+\overline{q_{ij}}\bigr)
\beta_{jk}
\gamma_{kl}
\Bigr\}
z_l.
$$
\end{lemma}
\begin{proof}
First we recall the definition of differentiation. 
For $z_j=x_j+iy_j\in\mathbb{C}$, $x_j,y_j\in\mathbb{R}$, 
$$
\frac{\partial}{\partial z_j}
=
\frac{1}{2}
\left(
\frac{\partial}{\partial x_j}
-
\sqrt{-1}
\frac{\partial}{\partial y_j}
\right),
\quad
\frac{\partial}{\partial \bar{z}_j}
=
\frac{1}{2}
\left(
\frac{\partial}{\partial x_j}
+
\sqrt{-1}
\frac{\partial}{\partial y_j}
\right).
$$
Let $F, G \in \mathscr{H}_\Phi(\mathbb{C}^n)$. 
Note that $\partial \bar{G}/\partial z_i=\overline{\partial G/\partial \bar{z}_i}=0$ 
since $G$ is holomorphic on $\mathbb{C}^n$. 
By using the integration by parts, we deduce that 
\begin{align}
  (\Lambda_iF,G)_{\mathscr{H}_\Phi(\mathbb{C}^n)}
& =
  \int_{\mathbb{C}^n}
  \left\{
  \frac{\partial F}{\partial z_i}(z)
  +
  2
  \sum_{j=1}^n
  q_{ij}z_j
  \right\}
  \overline{G(z)}
  e^{-2\Phi(z)}
  L(dz)
\nonumber
\\
& =
  \int_{\mathbb{C}^n}
  F(z)
  \overline{G(z)}
  \left\{
  \frac{\partial}{\partial z_i}
  \bigl(2\Phi(z)\bigr)
  +
  2
  \sum_{j=1}^n
  q_{ij}z_j
  \right\}
  e^{-2\Phi(z)}
  L(dz).
\label{equation:2101} 
\end{align}
Here we compute $\frac{\partial}{\partial z_i}\bigl(2\Phi(z)\bigr)$. 
An elementary computation yields 
\begin{align*}
  \frac{\partial}{\partial z_i}
  \bigl(2\Phi(z)\bigr)
& =
  \frac{\partial}{\partial z_i}
  \left\{
  2
  \sum_{j=1}^n
  \sum_{k=1}^n
  \alpha_{jk}
  z_j\bar{z}_k
  +
  \sum_{j=1}^n
  \sum_{k=1}^n
  \gamma_{jk}
  z_jz_k
  +
  \overline{
  \sum_{j=1}^n
  \sum_{k=1}^n
  \gamma_{jk}
  z_jz_k
  }
  \right\}
\\
& =
  2
  \sum_{j=1}^n
  \sum_{k=1}^n
  \alpha_{jk}
  \delta_{ij}\bar{z}_k
  +
  \sum_{j=1}^n
  \sum_{k=1}^n
  \gamma_{jk}
  \delta_{ij}z_k
  +
  \sum_{j=1}^n
  \sum_{k=1}^n
  \gamma_{jk}
  z_j\delta_{ik}
\\
& =
  2
  \sum_{k=1}^n
  \alpha_{ik}\bar{z}_k
  +
  \sum_{k=1}^n
  \gamma_{ik}z_k
  +
  \sum_{j=1}^n
  \gamma_{ji}z_j
\\
& =
  2
  \sum_{j=1}^n
  \alpha_{ij}\bar{z}_j
  +
  \sum_{j=1}^n
  \gamma_{ij}z_j
  +
  \sum_{j=1}^n
  \gamma_{ij}z_j
\\
& =
  2
  \sum_{j=1}^n
  \alpha_{ij}\bar{z}_j
  +
  2
  \sum_{j=1}^n
  \gamma_{ij}z_j. 
\end{align*}
Substitute this into \eqref{equation:2101}. 
We have 
\begin{equation}
(\Lambda_iF,G)_{\mathscr{H}_\Phi(\mathbb{C}^n)}
=
\int_{\mathbb{C}^n}
F(z)
\overline{G(z)}
\left\{
2
\sum_{j=1}^n
\alpha_{ij}\bar{z}_j
+
2
\sum_{j=1}^n
(\gamma_{ij}+q_{ij})z_j
\right\}
e^{-2\Phi(z)}
L(dz).
\label{equation:2102} 
\end{equation}
We need to create $z_je^{-2\Phi(z)}$ by using only $\frac{\partial}{\partial \bar{z}}e^{-2\Phi(z)}$ and $\bar{z}e^{-2\Phi(z)}$. 
Recall that $\Phi(z)$ is real-valued.
Here we note that 
\begin{align*}
  \frac{\partial}{\partial \bar{z}_k}
  e^{-2\Phi(z)}
& =
  \overline{\frac{\partial}{\partial z_k}e^{-2\Phi(z)}}
  =
  -
  \overline{\left(\frac{\partial}{\partial z_k}\bigl(2\Phi(z)\bigr)\right)}
  e^{-2\Phi(z)}
\\
& =
  -
  \left\{
  2
  \sum_{l=1}^n
  \overline{\alpha_{kl}}z_l
  +
  2
  \sum_{j=1}^n
  \overline{\gamma_{kl}}\bar{z}_l
  \right\}
  e^{-2\Phi(z)}. 
\end{align*}
Multiplying this by $\overline{\beta_{jk}}$ and summing up on $k$, we deduce that 
\begin{align*}
  \sum_{k=1}^n
  \overline{\beta_{jk}}
  \frac{\partial}{\partial \bar{z}_k}
  e^{-2\Phi(z)}
& =
  -
  2
  \sum_{k=1}^n
  \sum_{l=1}^n
  \overline{\beta_{jk}}
  \overline{\alpha_{kl}}z_l
  e^{-2\Phi(z)}
  -
  2
  \sum_{k=1}^n
  \sum_{l=1}^n
  \overline{\beta_{jk}}
  \overline{\gamma_{kl}}\bar{z}_l
  e^{-2\Phi(z)}
\\
& =
  -
  2
  \sum_{l=1}^n
  \delta_{jl}z_l
  e^{-2\Phi(z)}
  -
  2
  \sum_{k=1}^n
  \sum_{l=1}^n
  \overline{\beta_{jk}}
  \overline{\gamma_{kl}}\bar{z}_l
  e^{-2\Phi(z)}
\\
& =
  -
  2
  z_j
  e^{-2\Phi(z)}
  -
  2
  \sum_{k=1}^n
  \sum_{l=1}^n
  \overline{\beta_{jk}}
  \overline{\gamma_{kl}}\bar{z}_l
  e^{-2\Phi(z)}.
\end{align*}
By using this, we get 
\begin{align*}
  2
  \sum_{j=1}^n
  (\gamma_{ij}+q_{ij})
  z_j
  e^{-2\Phi(z)}
& =
  -
  \sum_{j=1}^n
  \sum_{k=1}^n
  (\gamma_{ij}+q_{ij})
  \overline{\beta_{jk}}
  \frac{\partial}{\partial \bar{z}_k}
  e^{-2\Phi(z)}
\\
& -
  2
  \sum_{j=1}^n
  \sum_{k=1}^n
  \sum_{l=1}^n
  (\gamma_{ij}+q_{ij})
  \overline{\beta_{jk}}
  \overline{\gamma_{kl}}\bar{z}_l
  e^{-2\Phi(z)}.
\end{align*}
Substituting this into \eqref{equation:2102}, we have 
\begin{align*}
  (\Lambda_iF,G)_{\mathscr{H}_\Phi(\mathbb{C}^n)}
& =
  \int_{\mathbb{C}^n}
  F(z)
  \overline{G(z)}
  \biggl\{
  -
  \sum_{j=1}^n
  \sum_{k=1}^n
  (\gamma_{ij}+q_{ij})
  \overline{\beta_{jk}}
  \frac{\partial}{\partial \bar{z}_k}
  e^{-2\Phi(z)}
\\
& \qquad \qquad
  +
  2
  \sum_{l=1}^n
  \left(
  \alpha_{il}
  -
  \sum_{j=1}^n
  \sum_{k=1}^n
  (\gamma_{ij}+q_{ij})
  \overline{\beta_{jk}}
  \overline{\gamma_{kl}}
  \right)
  \bar{z}_l
  e^{-2\Phi(z)}
  \biggr\}
  L(dz). 
\end{align*}
By using the integration by parts and $\partial F/\partial \bar{z}_k=0$, 
we prove Lemma~\ref{theorem:lemma2}.
\end{proof}
In order to consider (iii) of Condition~2, we compute a commutator $[\Lambda_i,\Lambda_j^\ast]$. We denote by $[\Lambda,\Lambda^\ast]$ an $n \times n$ matrix whose $(i,j)$-element is $[\Lambda_i,\Lambda_j^\ast]$. 
\begin{lemma}
\label{theorem:lemma2103} 
For $i,j=1,\dotsc,n$, 
$$
[\Lambda_i,\Lambda_j^\ast]
=
2
\left\{
\alpha_{ij}
-
\sum_{k=1}^n
\sum_{l=1}^n
(\gamma_{ik}+q_{ik})
\overline{\beta_{kl}}
(\overline{\gamma_{jl}}+\overline{q_{jl}})
\right\}, 
$$
that is, 
$$
[\Lambda,\Lambda^\ast]
=
2
\bigl\{
\Phi^{\prime\prime}_{z\bar{z}}
-
(\Phi^{\prime\prime}_{zz}+Q)
\overline{(\Phi^{\prime\prime}_{z\bar{z}})^{-1}}
(\Phi^{\prime\prime}_{zz}+Q)^\ast
\bigr\}.
$$
\end{lemma}
\begin{proof}
Let $F(z)$ be a holomorphic function on $\mathbb{C}^n$. 
We deduce that 
\begin{align*}
  [\Lambda_i,\Lambda_j^\ast]F
& =
  \left(
  \frac{\partial}{\partial z_i}
  +
  2
  \sum_{\mu=1}^n
  q_{i\mu}z_\mu
  \right)
\\
& \times
  \left\{
  \sum_{k=1}^n
  \sum_{l=1}^n
  (\overline{\gamma_{jl}}+\overline{q_{jl}})
  \beta_{lk}
  \frac{\partial F}{\partial z_k}
  +
  2
  \sum_{l=1}^n
  \overline{\alpha_{jl}}
  z_l
  F
  -
  2
  \sum_{k=1}^n
  \sum_{l=1}^n
  \sum_{m=1}^n
  (\overline{\gamma_{jl}}+\overline{q_{jl}})
  \beta_{lk}
  \gamma_{km}
  z_m
  F
  \right\} 
\\
& -
  \left\{
  \sum_{k=1}^n
  \sum_{l=1}^n
  (\overline{\gamma_{jl}}+\overline{q_{jl}})
  \beta_{lk}
  \frac{\partial}{\partial z_k}
  +
  2
  \sum_{l=1}^n
  \overline{\alpha_{jl}}
  z_l
  -
  2
  \sum_{k=1}^n
  \sum_{l=1}^n
  \sum_{m=1}^n
  (\overline{\gamma_{jl}}+\overline{q_{jl}})
  \beta_{lk}
  \gamma_{km}
  z_m
  \right\} 
\\
& \times
  \left(
  \frac{\partial F}{\partial z_i}
  +
  2
  \sum_{\mu=1}^n
  q_{i\mu}
  z_\mu
  F
  \right)
\\
& =
  \biggl\{
  2
  \sum_{l=1}^n
  \overline{\alpha_{jl}}
  \delta_{il}
  -
  2
  \sum_{k=1}^n
  \sum_{l=1}^n
  \sum_{m=1}^n
  (\overline{\gamma_{jl}}+\overline{q_{jl}})
  \beta_{lk}
  \gamma_{km}
  \delta_{im}
\\
& \qquad\qquad
  -
  2
  \sum_{k=1}^n
  \sum_{l=1}^n
  \sum_{\mu=1}^n
  (\overline{\gamma_{jl}}+\overline{q_{jl}})
  \beta_{lk}
  q_{i\mu}
  \delta_{k\mu}
  \biggr\}
  F
\\
& =
  \left\{
  2
  \overline{\alpha_{ji}}
  -
  2
  \sum_{k=1}^n
  \sum_{l=1}^n
  (\overline{\gamma_{jl}}+\overline{q_{jl}})
  \beta_{lk}
  \gamma_{ki}
  -
  2
  \sum_{k=1}^n
  \sum_{l=1}^n
  (\overline{\gamma_{jl}}+\overline{q_{jl}})
  \beta_{lk}
  q_{ik}
  \right\}
  F
\\
& =
  2
  \left\{
  \alpha_{ij}
  -
  \sum_{k=1}^n
  \sum_{l=1}^n
  (\gamma_{ik}+q_{ik})
  \overline{\beta_{kl}}
  (\overline{\gamma_{jl}}+\overline{q_{jl}})
  \right\}
  F. 
\end{align*}
This completes the proof.  
\end{proof}
We study (iii) of Condition~2. Recall that $\Phi^{\prime\prime}_{z\bar{z}}$ is a positive definite Hermitian matrix. There exists an $n \times n$ unitary matrix $U$ such that 
$$
\operatorname{diag}(\lambda_1^2, \dotsc, \lambda_n^2)
:=
\begin{bmatrix}
\lambda_1^2 & 0 & \dotsb & 0
\\
0 & \ddots & \ddots & \vdots
\\
\vdots & \ddots & \ddots & 0
\\
0 & \dotsb & 0 & \lambda_n^2 
\end{bmatrix}
=
U^\ast\Phi^{\prime\prime}_{z\bar{z}}U.
$$
In this case 
$$
\Phi^{\prime\prime}_{z\bar{z}}
=
U\operatorname{diag}(\lambda_1^2, \dotsc, \lambda_n^2)U^\ast, 
\quad
(\Phi^{\prime\prime}_{z\bar{z}})^{-1}
=
U\operatorname{diag}\left(\frac{1}{\lambda_1^2}, \dotsc, \frac{1}{\lambda_n^2}\right)U^\ast.
$$
Note that (iii) of Condition~2 is equivalent to the condition that there exists a positive constant $\rho$ such that 
\begin{equation}
\Phi^{\prime\prime}_{z\bar{z}}-\rho^2E
=
(\Phi^{\prime\prime}_{zz}+Q)
\overline{(\Phi^{\prime\prime}_{z\bar{z}})^{-1}}
(\Phi^{\prime\prime}_{zz}+Q)^\ast.
\label{equation:2201}
\end{equation}
Note that the both hand sides of \eqref{equation:2201} are Hermitian matrices, and their eigenvalues are all real-valued. 
If we set $\mu_i^2:=\lambda_i^2-\rho^2 \in \mathbb{R}$ for $i=1, \dotsc ,n$, then 
$$
\Phi^{\prime\prime}_{z\bar{z}}-\rho^2E
=
U
\operatorname{diag}(\mu_1^2, \dotsc, \mu_n^2)
U^\ast. 
$$
Note that $\Phi^{\prime\prime}_{zz}+Q$ is a complex symmetric matrix. 
The right hand side of \eqref{equation:2201} is nonnegative, that is, all the eigenvalues are nonnegative. Indeed, for any $\zeta\in\mathbb{C}^n$, we deduce that 
\begin{align*}
& {}^t\zeta
  (\Phi^{\prime\prime}_{zz}+Q)
  \overline{(\Phi^{\prime\prime}_{z\bar{z}})^{-1}}
  (\Phi^{\prime\prime}_{zz}+Q)^\ast
  \bar{\zeta} 
\\
  =
& {}^t\zeta
  (\Phi^{\prime\prime}_{zz}+Q)
  \overline{U}
  \operatorname{diag}\left(\frac{1}{\lambda_1^2}, \dotsc, \frac{1}{\lambda_n^2}\right)
  {}^tU
  (\Phi^{\prime\prime}_{zz}+Q)^\ast
  \bar{\zeta}
\\
  =
& {}^t\zeta
  (\Phi^{\prime\prime}_{zz}+Q)
  \overline{U}
  \operatorname{diag}\left(\frac{1}{\lambda_1}, \dotsc, \frac{1}{\lambda_n}\right)
  \operatorname{diag}\left(\frac{1}{\lambda_1}, \dotsc, \frac{1}{\lambda_n}\right)
  {}^tU
  (\Phi^{\prime\prime}_{zz}+Q)^\ast
  \bar{\zeta}
\\
  =
& {}^t 
  \left\{
  \operatorname{diag}\left(\frac{1}{\lambda_1}, \dotsc, \frac{1}{\lambda_n}\right)
  U^\ast
  (\Phi^{\prime\prime}_{zz}+Q)
  \zeta
  \right\}
  \overline{
  \left\{
  \operatorname{diag}\left(\frac{1}{\lambda_1}, \dotsc, \frac{1}{\lambda_n}\right)
  U^\ast
  (\Phi^{\prime\prime}_{zz}+Q)
  \zeta
  \right\}
  }
\\
  =
& \left\lvert
  \operatorname{diag}\left(\frac{1}{\lambda_1}, \dotsc, \frac{1}{\lambda_n}\right)
  U^\ast
  (\Phi^{\prime\prime}_{zz}+Q)
  \zeta
  \right\rvert^2
  \geqslant
  0.
\end{align*}
Hence all the $\mu_i^2, \dotsc, \mu_n^2$ must be nonnegative, 
and we may assume that all the $\mu_i, \dotsc, \mu_n$ are nonnegative. 
In other words, $\rho$ must satisfy $0<\rho\leqslant\lambda_0$.  
Thus the condition \eqref{equation:2201} becomes the following condition that 
there exists $\rho\in(0,\lambda_0]$ such that 
\begin{equation}
U
\operatorname{diag}(\mu_1^2, \dotsc, \mu_n^2)
U^\ast
=
(\Phi^{\prime\prime}_{zz}+Q)
\overline{(\Phi^{\prime\prime}_{z\bar{z}})^{-1}}
(\Phi^{\prime\prime}_{zz}+Q)^\ast.
\label{equation:2202}
\end{equation}
We see \eqref{equation:2202} as an equation for $Q$ for fixed $\rho\in(0,\lambda_0]$. 
We can characterize solutions $Q$ to \eqref{equation:2202} for $\rho \in (0,\lambda_0)$. 
In case of $\rho=\lambda_0$, the situation becomes complicated. 
As a result, the case of $\rho=\lambda_0$ will be useless in the next section, and we do not study this case. Note that if $0<\rho<\lambda_0$, then all the $\mu_i$ ($i=1,\dotsc,n$) are strictly positive. 
\begin{lemma}
\label{theorem:lemma4}
Fix an arbitrary $\rho\in(0,\lambda_0)$. 
The following conditions {\rm (I)} and {\rm (II)} are mutually equivalent.
\begin{itemize}
\item[{\rm (I)}] 
A complex symmetric matrix $Q$ solves \eqref{equation:2202}. 
\item[{\rm (II)}] 
There exists a unitary matrix $X$ such that 
\begin{equation}
Q
=
-
\Phi^{\prime\prime}_{zz}
+
U
\operatorname{diag}(\mu_1,\dotsc,\mu_n)
X
\operatorname{diag}(\lambda_1,\dotsc,\lambda_n)
{}^tU,
\label{equation:matrixQ} 
\end{equation}
\begin{equation}
{}^tX
\operatorname{diag}\left(\frac{\mu_1}{\lambda_1}, \dotsc, \frac{\mu_n}{\lambda_n}\right) 
=
\operatorname{diag}\left(\frac{\mu_1}{\lambda_1}, \dotsc, \frac{\mu_n}{\lambda_n}\right)
X.
\label{equation:matrixX} 
\end{equation}
\end{itemize}
\end{lemma}
Note that (II) implies (I) also for $\rho=\lambda_0$. 
A typical example of $X$ satisfying \eqref{equation:matrixX} is a diagonal matrix of the form 
$$
X
=
\operatorname{diag}(e^{\sqrt{-1}\theta_1}, \dotsc, e^{\sqrt{-1}\theta_n}), 
\quad
\theta_1, \dotsc, \theta_n \in \mathbb{R}.
$$
\begin{proof}[Proof of Lemma~\ref{theorem:lemma4}] 
\underline{(I) $\Rightarrow$ (II)}. 
Suppose that $Q$ solves \eqref{equation:2202}. 
Then we deduce that 
\begin{align*}
& \operatorname{diag}(\mu_1^2, \dotsc, \mu_n^2)
\\
  =
& U^\ast
  (\Phi^{\prime\prime}_{zz}+Q)
  \overline{U}
  \operatorname{diag}\left(\frac{1}{\lambda_1^2}, \dotsc, \frac{1}{\lambda_n^2}\right)
  {}^tU
  (\Phi^{\prime\prime}_{zz}+Q)^\ast
  U
\\
  =
& \left\{
  U^\ast
  (\Phi^{\prime\prime}_{zz}+Q)
  \overline{U}
  \operatorname{diag}\left(\frac{1}{\lambda_1}, \dotsc, \frac{1}{\lambda_n}\right)
  \right\}
  \left\{
  U^\ast
  (\Phi^{\prime\prime}_{zz}+Q)
  \overline{U}
  \operatorname{diag}\left(\frac{1}{\lambda_1}, \dotsc, \frac{1}{\lambda_n}\right)
  \right\}^\ast,   
\end{align*}
which becomes 
\begin{align*}
  E
& =
  \left\{
  \operatorname{diag}\left(\frac{1}{\mu_1}, \dotsc, \frac{1}{\mu_n}\right)
  U^\ast
  (\Phi^{\prime\prime}_{zz}+Q)
  \overline{U}
  \operatorname{diag}\left(\frac{1}{\lambda_1}, \dotsc, \frac{1}{\lambda_n}\right)
  \right\}
\\
& \times
  \left\{
  \operatorname{diag}\left(\frac{1}{\mu_1}, \dotsc, \frac{1}{\mu_n}\right)
  U^\ast
  (\Phi^{\prime\prime}_{zz}+Q)
  \overline{U}
  \operatorname{diag}\left(\frac{1}{\lambda_1}, \dotsc, \frac{1}{\lambda_n}\right)
  \right\}^\ast. 
\end{align*}
Then, there exists a unitary matrix $X$ such that 
$$
X
=
\operatorname{diag}\left(\frac{1}{\mu_1}, \dotsc, \frac{1}{\mu_n}\right)
U^\ast
(\Phi^{\prime\prime}_{zz}+Q)
\overline{U}
\operatorname{diag}\left(\frac{1}{\lambda_1}, \dotsc, \frac{1}{\lambda_n}\right). 
$$
This implies \eqref{equation:matrixQ}. 
We need to verify ${}^tQ=Q$. Note that ${}^tQ=Q$ and 
${}^t(\Phi^{\prime\prime}_{zz}+Q)=(\Phi^{\prime\prime}_{zz}+Q)$ are mutually equivalent since ${}^t\Phi^{\prime\prime}_{zz}=\Phi^{\prime\prime}_{zz}$. 
Moreover ${}^t(\Phi^{\prime\prime}_{zz}+Q)=(\Phi^{\prime\prime}_{zz}+Q)$ 
is equivalent to 
${}^t\{U^\ast(\Phi^{\prime\prime}_{zz}+Q)\overline{U}\}=U^\ast(\Phi^{\prime\prime}_{zz}+Q)\overline{U}$. 
Hence 
$\operatorname{diag}(\mu_1,\dotsc,\mu_n)X\operatorname{diag}(\lambda_1,\dotsc,\lambda_n)$ must be a complex symmetric matrix since \eqref{equation:matrixQ} becomes
$$
U^\ast
(\Phi^{\prime\prime}_{zz}+Q)
\overline{U}
=
\operatorname{diag}(\mu_1,\dotsc,\mu_n)
X
\operatorname{diag}(\lambda_1,\dotsc,\lambda_n). 
$$
Then $X$ must satisfy 
$$
\operatorname{diag}(\lambda_1,\dotsc,\lambda_n)
{}^t
X
\operatorname{diag}(\mu_1,\dotsc,\mu_n)
=
\operatorname{diag}(\mu_1,\dotsc,\mu_n)
X
\operatorname{diag}(\lambda_1,\dotsc,\lambda_n),
$$
which is equivalent to \eqref{equation:matrixX}. 
\vspace{6pt}
\\
\underline{(II) $\Rightarrow$ (I)}. 
Suppose that (II) holds. Then we deduce that 
\begin{align*}
& U^\ast
  (\Phi^{\prime\prime}_{zz}+Q)
  \overline{(\Phi^{\prime\prime}_{z\bar{z}})^{-1}}
  (\Phi^{\prime\prime}_{zz}+Q)^\ast
  U
\\
  =
& U^\ast
  (\Phi^{\prime\prime}_{zz}+Q)
  \overline{U}
  \operatorname{diag}\left(\frac{1}{\lambda_1^2}, \dotsc, \frac{1}{\lambda_n^2}\right)
  {}^tU
  (\Phi^{\prime\prime}_{zz}+Q)^\ast
  U
\\
  =
& U^\ast
  U
  \operatorname{diag}(\mu_1,\dotsc,\mu_n)
  X
  \operatorname{diag}(\lambda_1,\dotsc,\lambda_n)
  {}^tU
  \overline{U}
  \operatorname{diag}\left(\frac{1}{\lambda_1^2}, \dotsc, \frac{1}{\lambda_n^2}\right)
\\
  \times
& {}^tU
  \overline{U}
  \operatorname{diag}(\lambda_1,\dotsc,\lambda_n)
  X^\ast
  \operatorname{diag}(\mu_1,\dotsc,\mu_n)
  U^\ast
  U
\\
  =
& \operatorname{diag}(\mu_1,\dotsc,\mu_n)
  X
  \operatorname{diag}(\lambda_1,\dotsc,\lambda_n)
  \operatorname{diag}\left(\frac{1}{\lambda_1^2}, \dotsc, \frac{1}{\lambda_n^2}\right)
\\
  \times
& \operatorname{diag}(\lambda_1,\dotsc,\lambda_n)
  X^\ast
  \operatorname{diag}(\mu_1,\dotsc,\mu_n)
\\
  =
& \operatorname{diag}(\mu_1,\dotsc,\mu_n)
  X
  X^\ast
  \operatorname{diag}(\mu_1,\dotsc,\mu_n)
\\
  =
& \operatorname{diag}(\mu_1,\dotsc,\mu_n)
  \operatorname{diag}(\mu_1,\dotsc,\mu_n)
\\
  =
& \operatorname{diag}(\mu_1^2,\dotsc,\mu_n^2). 
\end{align*}
Then $Q$ solves \eqref{equation:2202}. 
This completes the proof.
\end{proof}
We see that Condition~1 gives no new restriction on $Q$. 
\begin{lemma}
\label{theorem:lemma5} 
Fix an arbitrary $\rho \in (0,\rho_0)$. 
If $Q$ solves \eqref{equation:2202}, then Condition~1 is automatically satisfied.  
\end{lemma}
\begin{proof}
Suppose that $Q$ solves \eqref{equation:2202}, that is, 
(II) of Lemma~\ref{theorem:lemma4} holds. 
For $z\ne0$, set 
$\zeta={}^t(\zeta_1,\dotsc,\zeta_n):={}^tUz$, 
and 
$$
\zeta_\lambda
:=
\operatorname{diag}(\lambda_1,\dotsc,\lambda_n)\zeta
=
{}^t(\lambda_1\zeta_1,\dotsc,\lambda_n\zeta_n), 
\quad
\zeta_\mu
:=
\operatorname{diag}(\mu_1,\dotsc,\mu_n)\zeta
=
{}^t(\mu_1\zeta_1,\dotsc,\mu_n\zeta_n).  
$$
Note that 
$\lvert{z}\rvert=\lvert\zeta\rvert$,   
$\lvert{X\zeta_\lambda}\rvert=\lvert\zeta_\lambda\rvert$ 
and 
$\lvert\zeta_\lambda\rvert\geqslant\lvert\zeta_\mu\rvert$.  
We deduce that 
\begin{align*}
& 2
  \langle{z,\Phi^{\prime\prime}_{z\bar{z}}\bar{z}}\rangle
  +
  2
  \operatorname{Re}\langle{z,(\Phi^{\prime\prime}_{zz}+Q)z}\rangle 
\\
  =
& 2
  \langle{z,U\operatorname{diag}(\lambda_1^2,\dotsc,\lambda_n^2)U^\ast\bar{z}}\rangle
\\
  +
& 2
  \operatorname{Re}\langle{z,U\operatorname{diag}(\mu_1,\dotsc,\mu_n)X\operatorname{diag}(\lambda_1,\dotsc,\lambda_n){}^tUz}\rangle 
\\
  =
& 2
  \langle{\zeta,\operatorname{diag}(\lambda_1^2,\dotsc,\lambda_n^2)\bar{\zeta}}\rangle
\\
  +
& 2
  \operatorname{Re}\langle{\zeta,\operatorname{diag}(\mu_1,\dotsc,\mu_n)X\operatorname{diag}(\lambda_1,\dotsc,\lambda_n)\zeta}\rangle 
\\
  =
& 2
  \langle{\zeta_\lambda,\overline{\zeta_\lambda}}\rangle
  +
  2
  \operatorname{Re}\langle{\zeta_\mu,X\zeta_\lambda}\rangle 
\\
  \geqslant
& 2
  \lvert\zeta_\lambda\rvert^2
  -
  2
  \lvert\zeta_\mu\rvert
  \lvert{X\zeta_\lambda}\rvert
\\
  =
& 2
  \lvert\zeta_\lambda\rvert^2
  -
  2
  \lvert\zeta_\mu\rvert
  \lvert\zeta_\lambda\rvert
\\
  =
& 2
  \lvert\zeta_\lambda\rvert
  (\lvert\zeta_\lambda\rvert-\lvert\zeta_\mu\rvert)
\\
  =
& \frac{2\lvert\zeta_\lambda\rvert}{\lvert\zeta_\lambda\rvert-\lvert\zeta_\mu\rvert}
  \bigl(\lvert\zeta_\lambda\rvert^2-\lvert\zeta_\mu\rvert^2\bigr)
\\
  \geqslant
& \lvert\zeta_\lambda\rvert^2-\lvert\zeta_\mu\rvert^2
\\
  =
& \sum_{i=1}^n 
  (\lambda_i^2-\mu_i^2)
  \lvert\zeta_i\rvert^2
\\
  =
& \rho^2
  \sum_{i=1}^n 
  \lvert\zeta_i\rvert^2
  =
  \rho^2\lvert\zeta\rvert^2, 
\end{align*}
which is desired. 
This completes the proof. 
\end{proof} 
Combining Lemmas~\ref{theorem:lemma1}, \ref{theorem:lemma4} and \ref{theorem:lemma5}, we have the results of the present section. 
\begin{theorem}
\label{theorem:theorem6} 
Fix an arbitrary $\rho \in (0,\rho_0)$. 
Conditions~1 and 2 hold if and only if there exists an 
$n \times n$ unitary matrix $X$ satisfying {\rm \eqref{equation:matrixX}} 
such that $Q$ is determined by {\rm \eqref{equation:matrixQ}}. 
In this case the anihilation operator $\Lambda$ is defined by 
$\Lambda=\partial/\partial z +2Qz$. 
\end{theorem}
%
%
\section{Properties}
\label{section:properties}
In what follows we fix arbitrary $\rho\in(0,\rho_0)$ 
and an arbitrary $n \times n$ unitary matrix $X$ satisfying \eqref{equation:matrixX}. 
Define $Q$ by \eqref{equation:matrixQ}, 
and set $\tilde{\psi}_0(z)=\exp(-\langle{z,Qz}\rangle)$ 
and $\Lambda=\partial/\partial z +2Qz$. 
We denote the set of nonnegative integers by $\mathbb{N}_0$. 
For a multi-index $\alpha=(\alpha_1,\dotsc,\alpha_n)\in\mathbb{N}_0^n$, 
set $\alpha!=\alpha_1!\dotsb\alpha_n!$, 
$\lvert\alpha\rvert=\alpha_1+\dotsb+\alpha_n$ and 
$$
\tilde{\psi}_\alpha(z)
:=
(\Lambda^\ast)^\alpha\tilde{\psi}_0(z)
=
(\Lambda^\ast_1)^{\alpha_1}\dotsb(\Lambda^\ast_n)^{\alpha_n}\tilde{\psi}_0(z). 
$$
The last one is well-defined since $[\Lambda^\ast_i,\Lambda^\ast_j]=0$. 
It follows that 
$\{\tilde{\psi}_\alpha\}_{\alpha\in\mathbb{N}_0^n} \subset \mathscr{H}_\Phi(\mathbb{C}^n)$ 
since 
Lemma~\ref{theorem:lemma5} implies that 
$\lvert\tilde{\psi}_0(z)\rvert^2e^{-2\Phi(z)} \leqslant e^{-\rho^2\lvert{z}\rvert^2}$. 
In the present section we study the properties of 
$\{\tilde{\psi}_\alpha\}_{\alpha\in\mathbb{N}_0^n}$. 
\par
We first check that $\{\tilde{\psi}_\alpha\}_{\alpha\in\mathbb{N}_0^n}$ is an orthogonal system of $\mathscr{H}_\Phi(\mathbb{C}^n)$ and a family of eigenfunctions for some Hamiltonian like that of the harmonic oscillator on $\mathbb{R}^n$. 
Set $X_i=\Lambda_i\Lambda_i^\ast$ and $Y_i=\Lambda_i^\ast\Lambda_i$ for $i=1,\dotsc,n$. 
Then we have 
$$
X_i=Y_i+2\rho^2, 
\quad
Y_i\tilde{\psi}_0=0, 
\quad
X_i\Lambda_i=\Lambda_i=\Lambda_iY_i, 
\quad
\Lambda_i^\ast X_i=Y_i \Lambda_i^\ast. 
$$
For our purpose, we need the following computation. We say that 
$$
\beta<\alpha, 
\quad
\text{for}
\quad
\alpha=(\alpha_1,\dotsc,\alpha_n), 
\beta=(\beta_1,\dotsc,\beta_n)
$$
if $\beta_i<\alpha_i$ for all $i=1,\dotsc,n$. 
\begin{lemma}
\label{theorem:computation} 
For $i=1,\dotsc,n$, $m=1,2,3,\dotsc$ and $\alpha\ne(0,\dotsc,0)$, we have 
\begin{align}
  \Lambda_i(\Lambda_i^\ast)^m
& =
  (\Lambda_i^\ast)^m\Lambda_i
  +
  2m\rho^2(\Lambda_i^\ast)^{m-1},
\label{equation:computation1}
\\
  \Lambda_i^m(\Lambda_i^\ast)^m
& =
  \prod_{k=1}^m
  (Y_i+2k\rho^2), 
\label{equation:computation2}
\\
  \Lambda_i^{m+1}(\Lambda_i^\ast)^m
& =
  \left\{
  \prod_{k=1}^m
  (X_i+2k\rho^2)
  \right\}
  \Lambda_i, 
\label{equation:computation3}
\\
  \Lambda^\alpha(\Lambda^\ast)^\alpha
& =
  Y^\alpha
  +
  \sum_{0<\beta<\alpha}
  C_\beta Y^\beta
  +
  (2\rho^2)^{\lvert\alpha\rvert}\alpha!
\label{equation:computation4}
\end{align}
where $C_\beta\in\mathbb{C}$ for $0<\beta<\alpha$ are appropriate constants. 
\end{lemma}
\begin{proof} 
\underline{\eqref{equation:computation1}}.\ 
The commutator $[\Lambda_i,\Lambda_i^\ast]=2\rho^2$ shows that 
\eqref{equation:computation1} holds for $m=1$. 
Suppose that \eqref{equation:computation1} holds for some $m\in\mathbb{N}$. 
We deduce that 
\begin{align*}
  \Lambda_i(\Lambda_i^\ast)^{m+1}
& =
  \{\Lambda_i\Lambda_i^\ast\}(\Lambda_i^\ast)^m
\\
& =
  \{\Lambda_i^\ast\Lambda_i+2\rho^2\}(\Lambda_i^\ast)^m 
\\
& =
  \Lambda_i^\ast\{\Lambda_i(\Lambda_i^\ast)^m\}+2\rho^2(\Lambda_i^\ast)^m
\\
& =
  \Lambda_i^\ast
  \{
  (\Lambda_i^\ast)^m\Lambda_i
  +
  2m\rho^2(\Lambda_i^\ast)^{m-1}
  \}
  +
  2\rho^2(\Lambda_i^\ast)^m
\\
& =
  (\Lambda_i^\ast)^{m+1}\Lambda_i
  +
  2(m+1)\rho^2(\Lambda_i^\ast)^m, 
\end{align*}
which shows that \eqref{equation:computation1} holds for $m+1$. 
The induction arguments on $m$ prove that 
\eqref{equation:computation1} holds for all $m\in\mathbb{N}$. 
\\
\underline{\eqref{equation:computation2}}.\ 
An elementary computation gives for $k\in\mathbb{N}_0$
\begin{equation}
\Lambda_i(Y_i+2k\rho^2)
=
\Lambda_iY_i+2k\rho^2\Lambda_i
=
(X_i+2k\rho^2)\Lambda_i
=
\bigl(Y_i+2(k+1)\rho^2\bigr)\Lambda_i. 
\label{equation:XY}
\end{equation}
The commutator $[\Lambda_i,\Lambda_i^\ast]=2\rho^2$ shows that 
\eqref{equation:computation2} holds for $m=1$. 
Suppose that \eqref{equation:computation2} holds for some $m\in\mathbb{N}$. 
Repeating \eqref{equation:XY} $n$ times, we deduce that 
\begin{align*}
  \Lambda_i^{m+1}(\Lambda_i^\ast)^{m+1}
& =
  \Lambda_i
  \{\Lambda_i^m(\Lambda_i^\ast)^m\}
  \Lambda_i^\ast
\\
& =
  \Lambda_i
  \left\{
  \prod_{k=1}^m
  (Y_i+2k\rho^2) 
  \right\}
  \Lambda_i^\ast
\\
& =
  \Lambda_i
  (Y_i+2\rho^2)
  \left\{
  \prod_{k=2}^m
  (Y_i+2k\rho^2) 
  \right\}
  \Lambda_i^\ast
\\
& =
  (X_i+2\rho^2)
  \Lambda_i
  \left\{
  \prod_{k=2}^m
  (Y_i+2k\rho^2) 
  \right\}
  \Lambda_i^\ast
\\
& =
  \dotsb
\\
& =
  \left\{
  \prod_{k=1}^m
  (X_i+2k\rho^2) 
  \right\}
  \Lambda_i\Lambda_i^\ast
\\
& =
  \left\{
  \prod_{k=2}^{m+1}
  (Y_i+2k\rho^2) 
  \right\}
  (Y_i+2\rho^2)
\\
& =
  \prod_{k=1}^{m+1}
  (Y_i+2k\rho^2),  
\end{align*}
which shows that \eqref{equation:computation2} holds for $m+1$. 
The induction arguments on $m$ prove that 
\eqref{equation:computation2} holds for all $m\in\mathbb{N}$.
\\
\underline{\eqref{equation:computation3}}.\ 
One can prove \eqref{equation:computation3} 
in the same way as \eqref{equation:computation2}. 
We omit the detail.
\\
\underline{\eqref{equation:computation4}}.\ 
By using \eqref{equation:computation2}, we deduce that 
\begin{align*}
  \Lambda^\alpha(\Lambda^\ast)^\alpha
& =
  \prod_{i=1}^n
  \{\Lambda_i^{\alpha_i}(\Lambda_i^\ast)^{\alpha_i}\}
\\
& =
  \prod_{i=1}^n
  \left\{
  \prod_{k_i=1}^{\alpha_i}
  (Y_i+2k_i\rho^2)
  \right\}
\\
& =
  Y^\alpha
  +
  \sum_{0<\beta<\alpha}
  C_\beta Y^\beta
  +
  \prod_{i=1}^n
  \left\{
  \prod_{k_i=1}^{\alpha_i}
  (2k_i\rho^2)
  \right\}
\\
& =
  Y^\alpha
  +
  \sum_{0<\beta<\alpha}
  C_\beta Y^\beta
  +
  \prod_{i=1}^n
  \left\{
  (2\rho^2)^{\alpha_i}\alpha_i!
  \right\}
\\
& =
  Y^\alpha
  +
  \sum_{0<\beta<\alpha}
  C_\beta Y^\beta
  +
  (2\rho^2)^{\lvert\alpha\rvert}
  \alpha!.
\end{align*}
This completes the proof.  
\end{proof}
Here we introduce an operator $H_\rho$ defined by 
$$
H_\rho
=
\Lambda_1^\ast\Lambda_1+\dotsb+\Lambda_n^\ast\Lambda_n+rho^2, 
$$
which corresponds to the Hamiltonian of the harmonic oscillator of the form 
$$
\sum_{i=1}^n
\left(
-
\frac{\partial^2}{\partial x_i^2}
+
x_i^2
\right), 
\quad
x=(x_1,\dotsc,x_n)\in\mathbb{R}^n.
$$
The family of functions $\{\tilde{\psi}_\alpha\}$ 
is an orthogonal system of eigenfunctions of $H_\rho$. 
\begin{theorem}
\label{theorem:hamiltonian}
For $\alpha,\beta\in(\mathbb{N}_0)^n$, 
\begin{equation}
H_\rho\tilde{\psi}_\alpha
=
(2\lvert\alpha\rvert+1)\rho^2\tilde{\psi}_\alpha,
\label{equation:eigen} 
\end{equation}
\begin{equation}
(\tilde{\psi}_\alpha,\tilde{\psi}_\beta)_{\mathscr{H}_\Phi(\mathbb{C}^n)}
=
\delta_{\alpha\beta}
(2\rho^2)^{\lvert\alpha\rvert}
\alpha!
\lVert\tilde{\psi}_0\rVert_{\mathscr{H}_\Phi(\mathbb{C}^n)}^2. 
\label{equation:orthogonal}
\end{equation}
\end{theorem}
\begin{proof}
Recall $\Lambda_i\tilde{\psi}_0=0$. 
By using \eqref{equation:computation1} of Lemma~\ref{theorem:computation}, 
we deduce that 
\begin{align*}
  H_\rho\tilde{\psi}_\alpha
& =
  \sum_{i=1}^n
  \Lambda_i^\ast\Lambda_i(\Lambda^\ast)^\alpha
  \tilde{\psi}_0
  +
  \rho^2
  \tilde{\psi}_\alpha
\\
& =
  \sum_{i=1}^n
  \left\{
  \prod_{j \ne i}
  (\Lambda_j^\ast)^{\alpha_j}
  \right\}
  \Lambda_i^\ast\Lambda_i(\Lambda_i^\ast)^{\alpha_i}
  \tilde{\psi}_0
  +
  \rho^2
  \tilde{\psi}_\alpha
\\
& =
  \sum_{i=1}^n
  \left\{
  \prod_{j \ne i}
  (\Lambda_j^\ast)^{\alpha_j}
  \right\}
  \{(\Lambda_i^\ast)^{\alpha_i+1}\Lambda_i+2\alpha_i\rho^2(\Lambda_i^\ast)^{\alpha_i}\}
  \tilde{\psi}_0
  +
  \rho^2
  \tilde{\psi}_\alpha
\\
& =
  \sum_{i=1}^n
  \left\{
  \prod_{j \ne i}
  (\Lambda_j^\ast)^{\alpha_j}
  \right\}
  \{2\alpha_i\rho^2(\Lambda_i^\ast)^{\alpha_i}\}
  \tilde{\psi}_0
  +
  \rho^2
  \tilde{\psi}_\alpha
\\
& =
  \sum_{i=1}^n
  \{2\alpha_i\rho^2\} 
  (\Lambda^\ast)^{\alpha}
  \tilde{\psi}_0
  +
  \rho^2
  \tilde{\psi}_\alpha
\\
& =
  \sum_{i=1}^n
  \{2\alpha_i\rho^2\} 
  \tilde{\psi}_\alpha
  +
  \rho^2
  \tilde{\psi}_\alpha
\\
& =
  (2\lvert\alpha\rvert+1)\rho^2
  \tilde{\psi}_\alpha, 
\end{align*}
which is \eqref{equation:eigen}. 
We show \eqref{equation:orthogonal}. 
Suppose that $\alpha\ne\beta$. Then $\alpha_i\ne\beta_i$ for some $i=1,\dotsc,n$. 
Without loss of generality, we may assume that $\alpha_i\geqslant\beta_i+1$. 
By using \eqref{equation:computation3} of Lemma~\ref{theorem:computation}, 
we deduce that 
\begin{align*}
& (\tilde{\psi}_\alpha,\tilde{\psi}_\beta)_{\mathscr{H}_\Phi(\mathbb{C}^n)}
\\
  =
& \left(
  (\Lambda^\ast)^\alpha\tilde{\psi}_0,
  (\Lambda^\ast)^\beta\tilde{\psi}_0,
  \right)_{\mathscr{H}_\Phi(\mathbb{C}^n)}
\\
  =
& \left(
  \left\{
  \prod_{j=1}^n(\Lambda_j^\ast)^{\alpha_j}
  \right\}
  \tilde{\psi}_0,
  \left\{
  \prod_{k=1}^n(\Lambda_j^\ast)^{\beta_k}
  \right\}
  \tilde{\psi}_0,
  \right)_{\mathscr{H}_\Phi(\mathbb{C}^n)}
\\
  =
& \left(
  \left\{
  \prod_{j{\ne}i}(\Lambda_j^\ast)^{\alpha_j}
  \right\}
  (\Lambda_i^\ast)^{\alpha_i-\beta_i-1}
  \tilde{\psi}_0,
  \left\{
  \prod_{k{\ne}i}(\Lambda_j^\ast)^{\beta_k}
  \right\}
  \{\Lambda_i^{\beta_i+1}(\Lambda_i^\ast)^{\beta_i}\}
  \tilde{\psi}_0,
  \right)_{\mathscr{H}_\Phi(\mathbb{C}^n)}
\\
  =
& \left(
  \left\{
  \prod_{j{\ne}i}(\Lambda_j^\ast)^{\alpha_j}
  \right\}
  (\Lambda_i^\ast)^{\alpha_i-\beta_i-1}
  \tilde{\psi}_0,
  \left\{
  \prod_{k{\ne}i}(\Lambda_j^\ast)^{\beta_k}
  \right\}
  \left\{
  \prod_{l=1}^{\beta_i}
  (X_i+2k\rho^2)
  \right\}
  \Lambda_i
  \tilde{\psi}_0,
  \right)_{\mathscr{H}_\Phi(\mathbb{C}^n)}
\\
  =
& 0, 
\end{align*}
which is \eqref{equation:orthogonal} for $\alpha\ne\beta$. 
We show \eqref{equation:orthogonal} for $\alpha=\beta$. 
Recall $Y_i\tilde{\psi}_0=\Lambda_i^\ast\Lambda_i\tilde{\psi}_0=0$. 
By using \eqref{equation:computation4} of Lemma~\ref{theorem:computation}. 
we deduce that 
\begin{align*}
  (\tilde{\psi}_\alpha,\tilde{\psi}_\alpha)_{\mathscr{H}_\Phi(\mathbb{C}^n)}
& =
  \left(
  (\Lambda^\ast)^\alpha\tilde{\psi}_0,
  (\Lambda^\ast)^\alpha\tilde{\psi}_0,
  \right)_{\mathscr{H}_\Phi(\mathbb{C}^n)}
\\
& =
  \left(
  \tilde{\psi}_0,
  \Lambda^\alpha(\Lambda^\ast)^\alpha\tilde{\psi}_0,
  \right)_{\mathscr{H}_\Phi(\mathbb{C}^n)}
\\
& =
  \left(
  \tilde{\psi}_0,
  \left\{
  Y^\alpha
  +
  \sum_{0<\beta<\alpha}
  C_\beta Y^\beta
  +
  (2\rho^2)^{\lvert\alpha\rvert}\alpha!
  \right\}
  \tilde{\psi}_0,
  \right)_{\mathscr{H}_\Phi(\mathbb{C}^n)}
\\
& =
  \left(
  \tilde{\psi}_0,
  (2\rho^2)^{\lvert\alpha\rvert}\alpha!
  \tilde{\psi}_0,
  \right)_{\mathscr{H}_\Phi(\mathbb{C}^n)}
\\
& =
  (2\rho^2)^{\lvert\alpha\rvert}\alpha!
  \lVert\tilde{\psi}_0\rVert_{\mathscr{H}_\Phi(\mathbb{C}^n)}^2. 
\end{align*}
This completes the proof. 
\end{proof}
Here we normalize $\{\tilde{\psi}_\alpha\}$. 
Set 
$$
\psi_\alpha(z)
=
\frac{\tilde{\psi}_\alpha(z)}{\sqrt{(2\rho^2)^{\lvert\alpha\rvert}\alpha!}\lVert\tilde{\psi}_0\rVert_{\mathscr{H}_\Phi(\mathbb{C}^n)}}, 
\quad
\alpha\in\mathbb{N}_0^n.
$$
Then $\{\psi_\alpha\}$ becomes an orthonormal system of 
$\mathscr{H}_{\Phi}(\mathbb{C}^n)$. 
\par
We consider the Rodorigues formula for $\{\psi_\alpha\}$. 
Recall that $\mu_1,\dotsc,\mu_n>0$ since $0<\rho<\lambda_0$. Then we have   
$$
\det(\Phi_{zz}^{\prime\prime}+Q)
=
(\det{U})^2 
(\det{X}) 
(\lambda_1\dotsb\lambda_n)
(\mu_1\dotsb\mu_n)
\ne
0. 
$$
Let $\Xi$ be the principal part of $\Lambda^\ast$, that is, 
$$
\Xi
=
(\Phi_{zz}^{\prime\prime}+Q)
(\Phi_{z\bar{z}}^{\prime\prime})^{-1}
\frac{\partial}{\partial z}, 
\quad
\Xi_i
=
\sum_{j=1}^n
\sum_{k=1}^n
(\overline{\gamma_{ij}}+\overline{q_{ij}})
\beta_{jk}
\frac{\partial}{\partial z_k}, 
\quad
i=1,\dotsc,n. 
$$
Here we state the Rodrigues formula. 
\begin{theorem}
\label{theorem:rodrigues}
If we set 
\begin{equation}
S
=
\Phi_{zz}^{\prime\prime}
-
U
\operatorname{diag}
\left(\frac{\lambda_1^2}{\mu_1},\dotsc,\frac{\lambda_n^2}{\mu_n}\right)
X
\operatorname{diag}(\lambda_1,\dotsc,\lambda_n)
{}^tU, 
\label{equation:matrixS}
\end{equation}
then $S$ is a complex symmetric matrix and for $\alpha\in\mathbb{N}_0^n$
\begin{equation}
\tilde{\psi}_\alpha(z)
=
e^{\langle{z,Sz}\rangle}
\Xi^\alpha
e^{-\langle{z,(S+Q)z}\rangle}
= 
e^{\langle{z,Sz}\rangle}
\Xi^\alpha
\bigl(
e^{-\langle{z,Sz}\rangle}
\tilde{\psi}_0(z)
\bigr),
\label{equation:rodrigues} 
\end{equation}
that is, 
$$
\psi_\alpha(z)
=
\frac{e^{\langle{z,Sz}\rangle}}{\sqrt{(2\rho^2)^{\lvert\alpha\rvert}\alpha!}\lVert\tilde{\psi}_0\rVert_{\mathscr{H}_\Phi(\mathbb{C}^n)}}
\Xi^\alpha
e^{-\langle{z,(S+Q)z}\rangle}
=
\frac{e^{\langle{z,Sz}\rangle}}{\sqrt{(2\rho^2)^{\lvert\alpha\rvert}\alpha!}}
\Xi^\alpha
\bigl(
e^{-\langle{z,Sz}\rangle}
\psi_0(z)
\bigr).
$$
\end{theorem}
\begin{proof}
Firstly we check that a matrix $S$ defined by \eqref{equation:matrixS} is a symmetric matrix. By using \eqref{equation:matrixX}, we deduce that 
\begin{align*}
  {}^tS
& =
  {}^t
  \left\{
  \Phi_{zz}^{\prime\prime}
  -
  U
  \operatorname{diag}
  \left(\frac{\lambda_1^2}{\mu_1},\dotsc,\frac{\lambda_n^2}{\mu_n}\right)
  X
  \operatorname{diag}(\lambda_1,\dotsc,\lambda_n)
  {}^tU
  \right\}
\\
& =
  {}^t\Phi_{zz}^{\prime\prime}
  -
  U
  \operatorname{diag}(\lambda_1,\dotsc,\lambda_n)
  {}^tX
  \operatorname{diag}
  \left(\frac{\lambda_1^2}{\mu_1},\dotsc,\frac{\lambda_n^2}{\mu_n}\right)
  {}^tU
\\
& =
  {}^t\Phi_{zz}^{\prime\prime}
  -
  U
  \operatorname{diag}(\lambda_1,\dotsc,\lambda_n)
  \left\{
  {}^tX
  \operatorname{diag}
  \left(\frac{\lambda_1}{\mu_1},\dotsc,\frac{\lambda_n}{\mu_n}\right)
  \right\}
  \operatorname{diag}(\lambda_1,\dotsc,\lambda_n)
  {}^tU
\\
& =
  {}^t\Phi_{zz}^{\prime\prime}
  -
  U
  \operatorname{diag}(\lambda_1,\dotsc,\lambda_n)
  \left\{
  \operatorname{diag}
  \left(\frac{\lambda_1}{\mu_1},\dotsc,\frac{\lambda_n}{\mu_n}\right)
  X
  \right\}
  \operatorname{diag}(\lambda_1,\dotsc,\lambda_n)
  {}^tU
\\
& =
  {}^t\Phi_{zz}^{\prime\prime}
  -
  U
  \operatorname{diag}
  \left(\frac{\lambda_1^2}{\mu_1},\dotsc,\frac{\lambda_n^2}{\mu_n}\right)
  X
  \operatorname{diag}(\lambda_1,\dotsc,\lambda_n)
  {}^tU
\\
& =
  S.
\end{align*}
Next we find that a complex symmetric matrix $S=[s_{ij}]$ satisfying desired equation 
$$
\Lambda_i^\ast F(z)
=
e^{\langle{z,Sz}\rangle}
\Xi_i
\bigl(
e^{-\langle{z,Sz}\rangle}
F(z)
\bigr), 
\quad
F \in \operatorname{Hol}(\mathbb{C}^n)
$$
must be of the form given in \eqref{equation:matrixS}. 
Since 
$$
e^{\langle{z,Sz}\rangle}
\Xi_i
\bigl(
e^{-\langle{z,Sz}\rangle}
F(z)
\bigr)
=
\Xi_iF(z)
-
\bigl\{
\Xi_i
(\langle{z,Sz}\rangle)
\bigr\}
F(z), 
$$
\begin{align*}
  \Xi_i(\langle{z,Sz}\rangle)
& =
  \sum_{j=1}^n
  \sum_{k=1}^n
  (\overline{\gamma_{ij}}+\overline{q_{ij}})
  \beta_{jk}
  \frac{\partial}{\partial z_k}
  \left\{
  \sum_{l=1}^n\sum_{m=1}^n
  s_{lm}z_lz_m
  \right\}
\\
& =
  2
  \sum_{j=1}^n
  \sum_{k=1}^n
  \sum_{l=1}^m
  (\overline{\gamma_{ij}}+\overline{q_{ij}})
  \beta_{jk}
  s_{kl}
  z_l
\end{align*}
we shall find $S$ satisfying 
$$
-
\sum_{j=1}^n
\sum_{k=1}^n
(\overline{\gamma_{ij}}+\overline{q_{ij}})
\beta_{jk}
s_{kl}
=
\overline{\alpha_{il}}
-
\sum_{j=1}^n
\sum_{k=1}^n
\bigl(\overline{\gamma_{ij}}+\overline{q_{ij}}\bigr)
\beta_{jk}
\gamma_{kl}, 
\quad
i,l=1,\dotsc,n, 
$$
that is, 
$$
(\Phi_{zz}^{\prime\prime}+Q)^\ast
(\Phi_{z\bar{z}}^{\prime\prime})^{-1}
S
=
(\Phi_{zz}^{\prime\prime}+Q)^\ast
(\Phi_{z\bar{z}}^{\prime\prime})^{-1}
\Phi_{zz}^{\prime\prime}
-
\overline{\Phi_{z\bar{z}}^{\prime\prime}}.
$$
This implies that 
$$
S
=
\Phi_{zz}^{\prime\prime}
-
\Phi_{z\bar{z}}^{\prime\prime}
\{(\Phi_{zz}^{\prime\prime}+Q)^\ast\}^{-1}
\overline{\Phi_{z\bar{z}}^{\prime\prime}}.
$$
We will find more concrete form of $S$. 
Recall \eqref{equation:matrixQ} and 
\begin{align*}
  \Phi_{z\bar{z}}^{\prime\prime}
& =
  U
  \operatorname{diag}(\lambda_1^2,\dotsc,\lambda_n^2)
  U^\ast,
\\
  \overline{\Phi_{z\bar{z}}^{\prime\prime}}
& =
  \overline{U}
  \operatorname{diag}(\lambda_1^2,\dotsc,\lambda_n^2)
  {}^tU,
\\
  (\Phi_{zz}^{\prime\prime}+Q)^\ast
& =
  \overline{\Phi_{zz}^{\prime\prime}+Q}
\\
& =
  \overline{U}
  \operatorname{diag}(\mu_1,\dotsc,\mu_n)
  \overline{X}
  \operatorname{diag}(\lambda_1,\dotsc,\lambda_n)
  U^\ast,
\\
  \{(\Phi_{zz}^{\prime\prime}+Q)^\ast\}^{-1}
& =
  U
  \operatorname{diag}\left(\frac{1}{\lambda_1},\dotsc,\frac{1}{\lambda_n}\right)
  {}^tX
  \operatorname{diag}\left(\frac{1}{\mu_1},\dotsc,\frac{1}{\mu_n}\right)
  {}^tU.
\end{align*}
By using these identities, we deduce that 
\begin{align*}
  S
  -
  \Phi_{zz}^{\prime\prime}
& =
  -
  U
  \operatorname{diag}(\lambda_1^2,\dotsc,\lambda_n^2)
  U^\ast
\\
& \times
  U
  \operatorname{diag}\left(\frac{1}{\lambda_1},\dotsc,\frac{1}{\lambda_n}\right)
  {}^tX
  \operatorname{diag}\left(\frac{1}{\mu_1},\dotsc,\frac{1}{\mu_n}\right)
  {}^tU
\\
& \times
  \overline{U}
  \operatorname{diag}(\lambda_1^2,\dotsc,\lambda_n^2)
  {}^tU
  \bigr\}
\\
& =
  -
  U
  \operatorname{diag}(\lambda_1,\dotsc,\lambda_n)
  {}^tX
  \operatorname{diag}\left(\frac{\lambda_1^2}{\mu_1},\dotsc,\frac{\lambda_n^2}{\mu_n}\right)
  {}^tU
\\
& =
  -
  U
  \operatorname{diag}(\lambda_1,\dotsc,\lambda_n)
  \left\{
  {}^tX
  \operatorname{diag}\left(\frac{\lambda_1}{\mu_1},\dotsc,\frac{\lambda_n}{\mu_n}\right)
  \right\}
  \operatorname{diag}(\lambda_1,\dotsc,\lambda_n)
  {}^tU
\\
& =
  -
  U
  \operatorname{diag}(\lambda_1,\dotsc,\lambda_n)
  \left\{
  \operatorname{diag}\left(\frac{\lambda_1}{\mu_1},\dotsc,\frac{\lambda_n}{\mu_n}\right)
  X
  \right\}
  \operatorname{diag}(\lambda_1,\dotsc,\lambda_n)
  {}^tU
\\
& =
  -
  U
  \operatorname{diag}\left(\frac{\lambda_1^2}{\mu_1},\dotsc,\frac{\lambda_n^2}{\mu_n}\right)
  X
  \operatorname{diag}(\lambda_1,\dotsc,\lambda_n)
  {}^tU,  
\end{align*}
which is \eqref{equation:matrixS}. 
This completes the proof.
\end{proof}
We remark that 
\begin{align*}
  S+Q
& =
  (S-\Phi_{zz}^{\prime\prime})+(Q+\Phi_{zz}^{\prime\prime})
\\
& =
  U
  \operatorname{diag}\left(\mu_1-\frac{\lambda_1^2}{\mu_1},\dotsc,\mu_n-\frac{\lambda_n^2}{\mu_n}\right)
  X
  \operatorname{diag}(\lambda_1,\dotsc,\lambda_n)
  {}^tU
\\
& =
  U
  \operatorname{diag}\left(-\frac{\lambda_1^2-\mu_1^2}{\mu_1},\dotsc,-\frac{\lambda_n^2-\mu_n^2}{\mu_n}\right)
  X
  \operatorname{diag}(\lambda_1,\dotsc,\lambda_n)
  {}^tU
\\
& =
  U
  \operatorname{diag}\left(-\frac{\rho^2}{\mu_1},\dotsc,-\frac{\rho^2}{\mu_n}\right)
  X
  \operatorname{diag}(\lambda_1,\dotsc,\lambda_n)
  {}^tU
\\
& =
  -
  \rho^2
  U
  \operatorname{diag}\left(\frac{1}{\mu_1},\dotsc,\frac{1}{\mu_n}\right)
  X
  \operatorname{diag}(\lambda_1,\dotsc,\lambda_n)
  {}^tU.
\\
  \Xi
& =
  \overline{U}
  \operatorname{diag}(\mu_1,\dotsc,\mu_n)
  \overline{X}
  \operatorname{diag}\left(\frac{1}{\lambda_1},\dotsc,\frac{1}{\lambda_n}\right)
  U^\ast
  \frac{\partial}{\partial z}.
\end{align*}
\par
Finally we check that the orthonormal system $\{\psi_\alpha\}$ 
is complete in $\mathscr{H}_\Phi(\mathbb{C}^n)$. 
\begin{theorem}
\label{theorem:completeness}
Suppose that $0<\rho<\lambda_0$. 
Then $\{\psi_\alpha\}$ is a complete orthonormal system of 
$\mathscr{H}_\Phi(\mathbb{C}^n)$. 
\end{theorem}
\begin{proof}
It suffices to see the completeness. 
Let $\zeta={}^tUz$. 
Then $z=\overline{U}\zeta$, 
$\partial/\partial z=U\partial/\partial \zeta$, 
$\lvert{z}\rvert=\lvert\zeta\rvert$ and 
$\lvert\psi_0(z)\rvert^2e^{-2\Phi(z)} \leqslant e^{-2\rho^2\lvert\zeta\rvert^2}$. 
The Rodrigues formula \eqref{equation:rodrigues} shows that 
\begin{align*}
  \tilde{\psi}_\alpha(z)
& =
  \tilde{\psi_0}(z)
  e^{\langle{z,(Q+S)z}\rangle}
  \Xi^\alpha
  \Bigl(
  e^{\langle{z,(Q+S)z}\rangle}
  \Bigr),
\\
  \tilde{\psi}_\alpha(\overline{U}\zeta)
& =
  \tilde{\psi_0}(\overline{U}\zeta)
  \exp
  \left\{
  -
  \rho^2
  \left\langle
  \zeta,
  \operatorname{diag}\left(\frac{1}{\mu_1},\dotsc,\frac{1}{\mu_n}\right)
  X
  \operatorname{diag}(\lambda_1,\dotsc,\lambda_n)
  \zeta
  \right\rangle
  \right\}
\\
& \times
  \left(
  \overline{U}
  \operatorname{diag}(\mu_1,\dotsc,\mu_n)
  \overline{X}
  \operatorname{diag}\left(\frac{1}{\lambda_1},\dotsc,\frac{1}{\lambda_n}\right)
  \frac{\partial}{\partial \zeta}
  \right)^\alpha
\\
& \times
  \exp
  \left\{
  \rho^2
  \left\langle
  \zeta,
  \operatorname{diag}\left(\frac{1}{\mu_1},\dotsc,\frac{1}{\mu_n}\right)
  X
  \operatorname{diag}(\lambda_1,\dotsc,\lambda_n)
  \zeta
  \right\rangle
  \right\}. 
\end{align*}
Hence $ \tilde{\psi}_\alpha(\overline{U}\zeta)$ 
is a product of $\tilde{\psi}_0(\overline{U}\zeta)$ 
and a kind of a Hermite polynomial of $\zeta$ of degree $\alpha$. 
Recall that $\{({}^tUz)^\alpha\tilde{\psi}_0(z)\}$ 
is a complete orthogonal system of 
$\mathscr{H}_\Phi(\mathbb{C}^n)$. 
Thus $\{\psi_\alpha\}$ is also a complete orthogonal system of 
$\mathscr{H}_\Phi(\mathbb{C}^n)$ 
since all the $({}^tUz)^\alpha$ is a finite linear combination 
of the Hermite polynomials of ${}^tUz$. 
\end{proof}
%
%
\section{Examples}
\label{section:examples}
Finally we shall see the known examples 
from a point of view of the results developed in the present paper. 
Let $s$ be a parameter satisfying $0<s<1$. 
%
%
\subsection{One-dimensional example of van Eijndhoven and Meyers in \cite{EM}}
We shall understand the example in \cite{EM} in terms of our theory. 
To have the function space $\chi_s(\mathbb{C})$, we need to impose 
$$
\frac{\lvert{B}\rvert^2}{4C_I}
=
\Phi^{\prime\prime}_{z\bar{z}}
=
\frac{1-s^2}{4s}, 
\quad
-
\frac{B^2}{4C_I}
-
\frac{A}{2\sqrt{-1}}
=
\Phi^{\prime\prime}_{zz}
=
-
\frac{1+s^2}{4s}
$$
on complex numbers $A$, $B$ and $C$. 
In this case $\lambda_1^2=\lambda_0^2=(1-s^2)/4s$. 
If we choose 
$$
A=\frac{\sqrt{-1}}{s}, 
\quad
B=\pm\sqrt{-1}\sqrt{1-s^2}, 
\quad
C=t+\sqrt{-1}s
\quad
(t\in\mathbb{R}), 
$$
then $\chi_s(\mathbb{C})$ is realized 
as the image of the Bargmann-type transform of $L^2(\mathbb{R})$. 
See, e.g., \cite{chihara3}. 
Their holomorphic Hermite functions are interpreted as 
$$
\tilde{\psi}_k(z)
=
e^{z^2/2}
\left(
-
\frac{1-s}{1+s}
\frac{d}{dz}
\right)^k
e^{-z^2},
\quad
k=0,1,2\dotsc. 
$$
In other words, this case corresponds to 
$$
Q=S=\frac{1}{2},
\quad
\Lambda
=
\frac{d}{dz}+z, 
\quad
\Lambda^\ast
=
\frac{1-s}{1+s}
\left(
-
\frac{d}{dz}+2z
\right).
$$
More precisely, $Q$ and $S$ are firstly determined by the exponential functions 
$e^{z^2/2}$ and $e^{-z^2}$. Next the anihilation operator $\Lambda$ is determined 
since $\Lambda$ takes the form $d/dz+\dotsb$ in our settings. 
Finally the creation operator $\Lambda^\ast$ 
is determined as the adjoint of $\Lambda$ in $\chi_s(\mathbb{C})$. 
If we set 
$X=1$, 
$U=\pm\sqrt{-1}$ 
and 
$\rho^2=(1-s)/(1+s)$, 
we can understand that 
the pair of $\chi_s(\mathbb{C})$ and $\{\tilde{\psi}_k\}_{k=0}^\infty$ 
is one of the examples of the results of the present paper. 
Indeed, we can check 
$$
\mu_1^2
:=
\lambda_1^2-\rho^2
=
\frac{(1-s)^3}{4s(1+s)}>0, 
$$
$$
Q+\Phi^{\prime\prime}_{zz}
=
Q
-
\frac{B^2}{4C_I}
-
\frac{A}{2i}
=
-
\frac{(1-s)^2}{4s}
=
U\mu_1X\lambda_1{}^tU, 
$$
$$
S
=
\frac{1}{2}
=
\Phi^{\prime\prime}_{zz}
-
U\frac{\lambda_1^2}{\mu_1}X\lambda_1{}^tU. 
$$
%
%
\subsection{Two-dimensional example of G\'orska, Horzela and Szafraniec in \cite{GHS}}
Finally, we shall understand the example in \cite{GHS} in terms of our theory. 
Let $Z=(z,\zeta)\in\mathbb{C}^2$. 
To have the function space $\chi_s(\mathbb{C}^2)$, we need to impose 
\begin{alignat*}{2}
  \Phi^{\prime\prime}_{z\bar{z}}
& =
  \frac{BC_I^{-1}B^\ast}{4}
  =
  \frac{1-s^2}{8s}
  E, 
& \quad
  E
& =
  \begin{bmatrix}
  1 & 0 
  \\
  0 & 1
  \end{bmatrix},
\\
  \Phi^{\prime\prime}_{zz}
& =
  -
  \frac{BC_I^{-1}B}{4}
  -
  \frac{A}{2\sqrt{-1}}
  =
  -
  \frac{1+s^2}{8s}
  K, 
& \quad
  K
& =
  \begin{bmatrix}
  0 & 1 
  \\
  1 & 0
  \end{bmatrix}.
\end{alignat*}
on $2\times2$ matrices $A$, $B$ and $C$. 
In this case $\lambda_0^2=\lambda_1^2=\lambda_2^2=(1-s^2)/8s$. 
This is realized by 
$$
A
=
\frac{\sqrt{-1}}{4s}
\bigl\{
(1-s^2)E
+
(1+s^2)K
\bigr\},
\quad
B
=
\pm\sqrt{-1}\sqrt{1-s^2}E, 
\quad
C
=
2\sqrt{-1}sE
$$
for instance. 
Their holomorphic Hermite functions are interpreted as 
$$
\tilde{\psi}_{k,l}(z,\zeta)
=
e^{z\zeta/2}
\left(
-
\frac{1-s}{1+s}
\frac{\partial}{\partial \zeta}
\right)^k
\left(
-
\frac{1-s}{1+s}
\frac{\partial}{\partial z}
\right)^l
e^{-z\zeta},
\quad
k,l=0,1,2,\dotsc.
$$
More precisely, this case corresponds to 
$$
Q=S=\frac{K}{4},
\quad
\Lambda
=
E\frac{\partial}{\partial Z}
+
\frac{K}{2}Z,
\quad
\Lambda^\ast
=
\frac{1-s}{1+s}
\left(
K
\frac{\partial}{\partial Z}
+
\frac{E}{2}
Z
\right).
$$
If we set $X=K$, $U=\pm\sqrt{-1}E$ and $\rho^2=(1-s)/2(1+s)$, 
we can understand that 
the pair of $\chi_s(\mathbb{C}^2)$ and $\{\tilde{\psi}_{k,l}\}_{k,l=0}^\infty$ 
is also one of the examples of the results of the present paper. 
Indeed, we can check 
$$
\mu_i^2
:=
\lambda_i^2
-
\rho^2
=
\frac{(1-s)^3}{8s(1+s)}
>
0, 
\quad
i=1,2,
$$
$$
{}^tX
\operatorname{diag}
\left(
\frac{\mu_1}{\lambda_1},
\frac{\mu_2}{\lambda_2}
\right)
=
\frac{1-s}{1+s}K
=
\operatorname{diag}
\left(
\frac{\mu_1}{\lambda_1},
\frac{\mu_2}{\lambda_2}
\right)
X,
$$
$$
Q+\Phi^{\prime\prime}_{zz}
=
Q
-
\frac{B^2}{4C_I}
-
\frac{A}{2i}
=
-
\frac{(1-s)^2}{8s}
K
=
U
\operatorname{diag}(\mu_1,\mu_2)
X
\operatorname{diag}(\lambda_1,\lambda_2)
{}^tU,
$$
$$
S
=
\frac{K}{4}
=
\Phi^{\prime\prime}_{zz}
-
U
\operatorname{diag}
\left(
\frac{\lambda_1^2}{\mu_1},
\frac{\lambda_2^2}{\mu_2}
\right)
X
\operatorname{diag}(\lambda_1,\lambda_2)
{}^tU.
$$
%
%

\end{document}